\documentclass[reqno]{article}

\usepackage[utf8]{inputenc}
\usepackage{lmodern}
\usepackage{array}
\usepackage{graphicx}
\usepackage{amssymb, amsfonts, amsthm, amsmath}
\usepackage[english]{babel}
\usepackage[pdfborder={0 0 0}]{hyperref}
\usepackage{tikz}
\usetikzlibrary{decorations}
\usetikzlibrary{decorations.pathreplacing}
\tikzset{individu/.style={draw,thick}}
\usepackage{subfigure}

\usepackage{float}
\usepackage{comment}

\theoremstyle{plain}
\newtheorem{theorem}{Theorem}[section]
\newtheorem{corollary}[theorem]{Corollary}
\newtheorem{lemma}[theorem]{Lemma}

\theoremstyle{definition}

\theoremstyle{remark}
\newtheorem{remark}[theorem]{Remark}

\numberwithin{equation}{section}

\newcommand{\N}{\mathbb{N}}

\newcommand{\R}{\mathbb{R}}

\DeclareMathOperator{\E}{\mathbb{E}}

\renewcommand{\P}{\mathbb{P}}

\newcommand{\egaldistr}{{\overset{(d)}{=}}}

\newcommand{\calA}{\mathcal{A}}

\usepackage{tikz}

\makeatletter \newcommand \listoftodos{\section*{Todo list} \@starttoc{tdo}}
\newcommand\l@todo[2]
{\par\noindent \textit{#2}, \parbox{10cm}{#1}\par} \makeatother

\begin{document}
\newcommand{\asn}{\quad \text{a.s.}}
\newcommand{\eee}{\mathbb{E}}
\newcommand{\ppp}{\mathbb{P}}
\newcommand{\mo}{\mathcal{O}}
\newcommand{\pp}{\mathcal{P}}
\newcommand{\ip}{\ \text{in probability}}
\newcommand{\ipn}{\text{in probability}}
\newcommand{\yw}{\widetilde{Y}}
\newcommand{\bigo}{\mathcal{O}}
\newcommand{\smallo}{\mathcal{o}}
\newcommand{\at}{\mathcal{A}_{t,\epsilon}}
\newcommand{\inlaw}{\quad in \ law}
\newcommand{\as}{\quad  a.s.}
\newcommand{\asb}{\alpha\mathrm{-stable}}
\newcommand{\rhonu}{\rho_\nu(\cdot)}
\newcommand{\sm}{\mathcal{P}_{s}}
\newcommand{\mass}{\mathcal{P}_{m}}
\newcommand{\pk}{\mathcal{P}_{k}}
\newcommand{\pn}{\mathcal{P}_{n}}
\newcommand{\pin}{\mathcal{P}_{n} \backslash 0_{[n]}}
\newcommand{\pif}{\mathcal{P}_\infty}
\newcommand{\LC}{\Lambda\text{-coalescent}}
\newcommand{\ed}{\overset{D}{=}}
\newcommand{\mbN}{\mathbb{N}}
\newcommand{\mbP}{\mathbb{P}}
\newcommand{\mcP}{\mathcal{P}}
\newcommand{\Coag}{\mathrm{Coag}}
\newcommand{\limtau}{\lim_{\tau \rightarrow 0^+}}
\newcommand{\simt}{\underset{t \rightarrow 0^+}{\sim}}
\newcommand{\simk}{\underset{k \rightarrow \infty}{\sim}}
\newcommand{\simn}{\underset{n \rightarrow \infty}{\sim}}
\newcommand{\simq}{\underset{q \rightarrow +\infty}{\sim}}
\newcommand{\di}{\displaystyle}
\newcommand{\var}{\mathrm{Var}\,}
\newcommand{\sumnk}{\sum_{n=k}^{\infty}}
\newcommand{\sumnkk}{\sum_{n=k+1}^{\infty}}
\newcommand{\limk}{\lim_{k \rightarrow \infty}}
\newcommand{\limn}{\lim_{n \rightarrow \infty}}
\newcommand{\limt}{\lim_{t \rightarrow 0^+}}
\newcommand{\lims}{\lim_{s \rightarrow 0^+}}
\newcommand{\Beta}{\mathrm{Beta}}
\title{Second order behavior of the block counting process of beta coalescents}
\author{Yier Lin\thanks{The author would like to express his sincere thank to 'Tsinghua Xue Tang Program', which provides him funds and opportunity to do research in ENS.} and Bastien Mallein\footnote{DMA, ENS.}}
\maketitle

\begin{abstract}
The Beta coalescents are stochastic processes modeling the genealogy of a population. They appear as the rescaled limits of the genealogical trees of numerous stochastic population models. In this article, we take interest in the number of blocs at small times in the Beta coalescent. Berestycki, Berestycki and Schweinsberg \cite{BBS08} proved a law of large numbers for this quantity. Recently, Limic and Talarczyk \cite{LiT15} proved that a functional central limit theorem holds as well. We give here a simple proof for an unidimensional version of this result, using a coupling between Beta coalescents and continuous-time branching processes.
\end{abstract}

\section{Introduction}
\label{sec:intro}
A coalescent process is a stochastic model for the genealogy of an infinite haploid population, built backward in time. In such a model, an individual is represented by an integer $n \in \N$. At each time $t$, we denote by $\Pi(t)$ the partition of $\N$ such that two individuals $i$ and $j$ belong to the same set in $\Pi(t)$ (that we call ``bloc'' from now on) if they share a common ancestor less than $t$ units of time in the past. In particular, we always assume that $\Pi(0) = \{ \{ 1\},\{2\}, \ldots \}$ is the partition in singletons. We construct $(\Pi(t), t \geq 0)$ as a Markov process on the set of partitions, that gets coarser over time.

Let $\Lambda$ be a probability measure on $[0,1]$. The $\Lambda$-coalescent is a coalescent process such that given there are $b$ distinct blocs in $\Pi(t)$, any particular set of $k$ blocs merge at rate
\[
  \lambda_{b,k} = \int_0^1 x^{k-2}(1-x)^{b-k} \Lambda(dx).
\]
The $\Lambda$-coalescent has been introduced independently by Pitman \cite{Pit99} and Sagitov~\cite{Sag99}. In this process, several blocs may merge at once, but at most one such coalescing event may occur at a given time.

For any $t \geq 0$, we denote by $N(t)$ the number of blocs in $\Pi(t)$. We have in particular $N(0) = +\infty$. We say that the $\Lambda$-coalescent comes down from infinity if almost surely $N(t) < +\infty$ for any $t > 0$. Pitman \cite{Pit99} proved that if $\Lambda(\{1\})=0$, either the $\Lambda$-coalescent comes down from infinity, or $N(t)= +\infty$ for any $t >0$ a.s. In the rest of the article, we always assume that $\Lambda$ has no atom at 1.

Schweinsberg \cite{Sch00} obtained a necessary and sufficient condition for the $\Lambda$-coalescent to come down from infinity, that Bertoin and Le Gall \cite{BeG06} proved equivalent to
\begin{equation}
  \label{eqn:defPsi}
  \int_1^{+\infty} \frac{dq}{\psi(q)} < +\infty, \quad \text{where } \psi(q) = \int_0^1 (e^{-qx} - 1 + q x) x^{-2} \Lambda(dx).
\end{equation}
Berestycki, Berestycki and Limic \cite{BBL10} obtained the almost sure behaviour for the number of blocs $N(t)$ as $t$ goes to 0, which they called the speed of coming down from infinity. More precisely, setting $v_\psi(t) = \inf\{ s > 0 : \int_s^{+\infty} \frac{dq}{\psi(q)} \leq t \}$, they proved that for a $\Lambda$-coalescent that comes down from infinity,
\begin{equation}
  \label{eqn:speedCDI}
  \lim_{t \to 0} \frac{N(t)}{v_\psi(t)} = 1 \quad \text{a.s.}
\end{equation}

In this article, we consider the one parameter family of coalescent processes called Beta-coalescents. For any $\alpha \in (0,2)$, we consider the $\Lambda$-coalescent such that the measure $\Lambda$ is $\Beta(2-\alpha,\alpha)$, i.e.
\[
  \Lambda(dx) = \frac{1}{\Gamma(\alpha)\Gamma(2-\alpha)} x^{1-\alpha} (1-x)^{\alpha - 1} dx.
\]
The Beta-coalescents have a number of interesting properties (see e.g. \cite{BBC+05,BBS08} and references therein). In particular, if $\alpha \in (1,2)$, it can be constructed as the genealogy of an $\alpha$-stable continuous state branching process.

We observe that thanks to \eqref{eqn:defPsi}, $\alpha \in (1,2)$ is a necessary and sufficient condition for the Beta-coalescent to come down from infinity. Moreover, \eqref{eqn:speedCDI} can be restated as
\[
  \lim_{t \to 0} t^\frac{1}{\alpha-1} N(t) = (\alpha \Gamma(\alpha))^\frac{1}{\alpha - 1} \quad \text{a.s.}
\]
The speed of coming down from infinity for the Beta coalescent can also be found in \cite{BBS08}. The main result of this article is a central limit theorem for the number of blocs, as $t \to 0$.
\begin{theorem}
\label{thm:main}
Let $\alpha \in (1,2)$ we set $(\Pi(t),t \geq 0)$ the $\Beta(2-\alpha,\alpha)$-coalescent and $N(t) = \# \Pi(t)$ the number of blocs at time $t$, we have
\[
  \lim_{t \to 0} t^{\frac{1}{\alpha(\alpha-1)}} \left(N(t) - \left(\frac{\alpha \Gamma(\alpha)}{t}\right)^{\frac{1}{\alpha -1}} \right) = - D_\alpha X \quad \text{in law,}
\]
where $D_\alpha = \left(\Gamma(\alpha)\alpha \right)^{\frac{1}{\alpha(\alpha-1)}}(\alpha-1)^{-\frac{1}{\alpha}}$, $X=\int_0^1 Y(t) dt$ and  $(Y(t), t\geq 0)$ is a Lévy process satisfying $\E(e^{-\lambda Y_t}) = e^{t \lambda^\alpha}$.
\end{theorem}
Note that a more precise functional central limit theorem has been obtained by \cite{LiT15} for any $\Lambda$-coalescent with a regularly varying density in a neighbourhood of $0$. However, our proof follows from simple coupling arguments, that might be of independent interest.

\begin{remark}
We observe that the random variable $X$ defined in Theorem \ref{thm:main} is an $\alpha$-stable random variable, that satisfies
\[
  \E(e^{-\lambda X}) = \exp\left( \frac{\lambda^\alpha}{\alpha + 1} \right).
\]
\end{remark}

In Section \ref{sec:csbp}, we use \cite{BBS08} to couple the Beta-coalescent with a stable continuous state branching process, and link the small times behaviour of the number of blocs with the small times behaviour of the continuous-state branching process. In Section \ref{sec:lamperti}, we use the so-called Lamperti transform to transfer the computations into the small times asymptotic of an $\alpha$-stable Lévy process, and use scaling properties to conclude.

\section{Continuous state branching process}
\label{sec:csbp}
A continous-state branching process (or CSBP for short) is a càdlàg (right-continuous with left limits at each point) Markov process $(Z(t), t \geq 0)$ on $\R_+$ that satisfies the so-called branching property: For any $x, y \geq 0$, if $(Z_x(t), t \geq 0)$ and $(Z_y(t), t \geq 0)$ are two independent versions of $Z$ starting from $x$ and $y$ respectively, then the process $(Z_x(t) + Z_y(t), t \geq 0)$ is also a version of $Z$ starting from $x+y$.

The study of CSBP started with the seminal work of \cite{Jir58}. As observed in \cite{Lam67,Sil67}, there exists a deep connexion between CSBP and Lévy processes. In effect, we observe that for any $x,t, \lambda \geq 0$, the Laplace transform of the CSBP $Z$ satisfies
\[
  \E\left( \exp(-\lambda Z_x(t) \right) = \exp(- x u_t(\lambda)),
\]
where $u$ is the solution of the following differential equation
\begin{equation}
  \label{eqn:csbpLevy}
  \partial_t u_t(\lambda) = \phi(u_t(\lambda)), \quad \text{with } u_0(\lambda) = \lambda,
\end{equation}
and $\phi$ is the Lévy-Khinchine exponent of a spectrally positive Lévy process (i.e. a Lévy process with no negative jump). The function $\phi$ is called the branching mechanism of the CSBP. If $\phi : \lambda \mapsto \lambda^\alpha$ with $\alpha \in (1,2)$, we call $Z$ the $\alpha$-stable CSBP.

Let $\alpha \in (1,2)$. Berestycki, Berestycki and Schweinsberg gave in \cite{BBS08} a coupling between the $\alpha$-stable CSBP and the $\Beta(2-\alpha,\alpha)$-coalescent, that we recall here. Let $(Z_a(t), t \geq 0, a \in [0,1])$ be a random field, càdlàg in $t$ and $a$, such that for any  $a < b$, the process $(Z_b(t)-Z_a(t), t \geq 0)$ is the $\alpha$-stable CSBP starting from $b-a$, and is independent with $(Z_c(t), t \geq 0, c < a)$. For any $t > 0$, the function $a \mapsto Z_a(t)$ is a.s. increasing, and we set
\begin{equation}
  \label{eqn:defD}
  D(t) = \# \left\{ a \in (0,1) :  Z_{a-}(t) < Z_a(t) \right\}
\end{equation}
the number of atoms in the measure $\mu_t$ satisfying $\mu_t([0,a]) = Z_a(t)$ a.s.

We also introduce $R(t) = C_\alpha \int_0^t Z_1(s)^{1-\alpha} dt$, where $C_\alpha = \alpha(\alpha-1)\Gamma(\alpha)$, as well as its generalized inverse
\begin{equation}
  \label{eqn:defRm1}
  R^{-1}(t) = \inf\left\{ s \geq 0 : R(s) > t \right\}.
\end{equation}
The coupling between the CSBP and the Beta-coalescent is obtained as a straightforward combination of Lemmas 2.1 and 2.2 in \cite{BBS08}.
\begin{lemma}[\cite{BBS08}]
\label{lem:coupling}
For any $t > 0$, we have $N(t) \egaldistr D(R^{-1}(t))$.
\end{lemma}

Using this result, to compute the small times behaviour of $N(t)$, it is enough to study the asymptotic behaviour of $D(r)$ and $R^{-1}(t)$ separately. We first provide a straightforward estimate on the asymptotic behaviour of $D$.
\begin{theorem}\label{thm:poisson}
For any $\alpha \in (1,2)$, for any $\epsilon>0$, we have
\[
  \lim_{r \to 0} \frac{D(r) - ((\alpha-1)r)^{-\frac{1}{\alpha -1}}}{r^{-\frac{1}{2(\alpha-1)}-\epsilon}} = 0 \quad \text{a.s.}
\]
\end{theorem}

\begin{proof}
We note that $(D(r), r > 0)$ is decreasing. Moreover, for any $r \geq 0$, $D(r)$ is a Poisson random variable with parameter $\theta_r = ((\alpha-1)r)^{-\frac{1}{\alpha-1}}$, by Lemma~2.2 of~\cite{BBS08}. Therefore, by a deterministic change of variables, it is enough to observe that for any increasing process $(P(t), t \geq 0)$ such that $P(t)$ is a Poisson random variable with parameter $t$, we have
\[
  \lim_{t \to +\infty} \frac{P(t)-t}{t^{\frac{1}{2}+\epsilon}} = 0 \quad \text{a.s.}
\]

Using the exponential Markov inequality, for any $\lambda > 0$ we have
\[
  \P(P(t) - t > t^{\frac{1}{2} + \epsilon}) \leq e^{-\lambda t^{\frac{1}{2}+\epsilon}} \E\left( e^{\lambda(P(t)-t)} \right) = \exp\left( t (e^{\lambda}-1-\lambda) - \lambda t^{\frac{1}{2} + \epsilon} \right).
\]
Applying this inequality with $\lambda = t^{-1/2}$, there exists $C_\epsilon>0$ such that for any $t \geq 1$, $\P(P(t) - t > t^{\frac{1}{2} + \epsilon}) \leq C_\epsilon e^{-t^\epsilon}$. With similar computations, we have
\[
  \P(P(t) - t < - t^{\frac{1}{2} + \epsilon}) \leq C_\epsilon e^{-t^{\epsilon}}.
\]
We apply the Borel-Cantelli lemma, yielding $\limsup_{n \to + \infty} \frac{|P(n)-n|}{n^{\frac{1}{2}+\epsilon}} \leq 1$ a.s. As $P$ is increasing, we obtain that for any $\epsilon>0$, $\lim_{t \to +\infty} \frac{P(t)-t}{t^{\frac{1}{2}+\epsilon}} = 0$ a.s. concluding the proof.
\end{proof}

\section{The Lamperti transform}
\label{sec:lamperti}
The connexion between CSBP and spectrally positive Lévy processes observed in \eqref{eqn:csbpLevy} can be strengthen. In \cite{Lam67}, Lamperti observed that a CSBP with branching mechanism $\phi$ could be constructed as a random time change of a Lévy process with Lévy-Khinchine exponent $\phi$. A proof of this result can be found in \cite{CLUB}. More precisely, let $(Y(t), t \geq 0)$ be a spectrally positive Lévy process starting from $a$, such that $\E(e^{-\lambda Y(t)}) = e^{-a\lambda+t \phi(\lambda)}$. We set $T = \inf\{ s \geq 0 : Y(s) \leq 0 \}$ and
\[
  U(t) = \inf\left\{ s \geq 0 : \int_0^s \frac{dr}{Y(r \wedge T)} > t \right\}.
\]
The Lamperti transform states that for $Z$ a CSBP with branching mechanism $\phi$ such that $Z(0)=a$, we have
\begin{equation}
  \label{eqn:lampertiTransform}
  \left( Z(t), t \geq 0 \right) \egaldistr \left( Y(U(t)), t \geq 0 \right)
\end{equation}

In the rest of the section, we denote by $(Y(t), t \geq 0)$ a Lévy process with Lévy-Khinchine exponent $\phi(\lambda) = \lambda^\alpha$ such that $Y(0)=1$ a.s. We also set $Y_0(t) = Y(t)-1$. We write $T = \inf\left\{ s \geq 0 : Y(s) \leq 0 \right\}$ and
\[
  U(t) = \inf\left\{ s \geq 0 : \int_0^s \frac{du}{Y(u \wedge T)} \geq t \right\}.
\] 
Using \eqref{eqn:lampertiTransform}, the process defined in \eqref{eqn:defRm1} satisfies
\begin{equation}
  \label{eqn:applLamperti}
  \left(R^{-1}(t), t \geq 0 \right) \egaldistr \left( \inf\left\{ s \geq 0 : C_\alpha\int_0^s Y(U(u))^{1-\alpha} du \geq t\right\}, t \geq 0 \right).
\end{equation}
Therefore, up to a slight abuse of notation, we write
\begin{equation}
  \label{eqn:defR}
  R(t) = C_\alpha\int_0^t Y(U(s))^{1-\alpha} ds = C_\alpha \int_0^{U(t)} Y(u)^{-\alpha} du,
\end{equation}
by change of variable, and again $R^{-1}(t) = \inf\left\{ s \geq 0 : R(s) \geq t \right\}$. We first prove a central limit theorem for the asymptotic behaviour of $R(t)$ as $t \to 0$.
\begin{theorem}
\label{thm:asymptoticR}
We denote by $X = \int_0^1 Y_0(s) ds$. We have
\[
  \lim_{t \to 0} \frac{R(t) - C_\alpha t}{t^{1 + \frac{1}{\alpha}}} = (1-\alpha)C_\alpha X \quad \text{in law}.
\]
\end{theorem}

\begin{proof}
For any $\epsilon>0$ and $t > 0$, we write $\calA_{t,\epsilon} = \{ |Y(s)-1| \leq \epsilon, s \leq 2t \}$ the event such that $Y$ stays in an $\epsilon$ neighbourhood of 1 until time $2t$. As observed in \cite[Lemma 4.2]{BBS08}, there exists $C>0$ such that $\P(\calA_{t,\epsilon}^c) \leq C t \epsilon^{-\alpha}$.

We first prove that $\lim_{t \to 0} \frac{U(t)}{t} = 1$ and $\lim_{t \to 0} \frac{R(t)}{t} = C_\alpha$ a.s. Let $\epsilon<1/2$, observe that on the event $\calA_{t,\epsilon}$, we have $T > 2t$, therefore for any $s \leq t$, we have
\[
  U(s) = \inf\left\{ r \geq 0 : \int_0^r \frac{du}{Y(u)} \geq s \right\} \in \left[ \tfrac{s}{1+\epsilon}, \tfrac{s}{1-\epsilon}\right].
\]
In particular, letting $t \to 0$ we obtain
\[
  \frac{1}{1+\epsilon} \leq \liminf_{s \to 0} \frac{U(s)}{s} \leq \limsup_{s \to 0} \frac{U(s)}{s} \leq \frac{1}{1-\epsilon} \quad \text{a.s.}
\]
Letting $\epsilon \to 0$, this yields $\lim_{t \to 0} \frac{U(s)}{s} = 1$ a.s. Similarly, by \eqref{eqn:defR} we have
\[
  \frac{1}{(1+\epsilon)^{1+\alpha}} \leq \liminf_{s \to 0} \frac{R(s)}{C_\alpha s} \leq \limsup_{s \to 0} \frac{R(s)}{C_\alpha s} \leq \frac{1}{(1-\epsilon)^{1+\alpha}},
\]
yielding $\lim_{t \to 0} \frac{R(t)}{t} = C_\alpha$ a.s.

We set $\tilde{R}(t) = R(t) - C_\alpha t$, we have
\[
  \tilde{R}(t) = C_\alpha \int_0^{U(t)} \left(Y(s)^{-\alpha} - \frac{1}{Y(s)}\right) ds  = C_\alpha \int_0^{U(t)} \frac{(1 + Y_0(s))^{1-\alpha} - 1}{1+Y_0(s)}ds. 
\]
As a consequence, we have
\begin{equation}
  \label{eqn:firstDecomposition}
  \tilde{R}(t) = C_\alpha(1-\alpha) \int_0^{U(t)} Y_0(s) ds + \Delta(t),
\end{equation}
where $\Delta(t) = C_\alpha \int_0^{U(t)} \frac{(1 + Y_0(s))^{1-\alpha}-1 -(1-\alpha)Y_0(s) - (1-\alpha)Y_0(s)^2}{1 + Y_0(s)} ds$. Note that as $Y_0$ is an $\alpha$-stable Lévy process, the following scaling property holds for any $\lambda > 0$:
\begin{equation}
  \label{eqn:scaling}
  \left(Y_0(t),t \geq 0 \right) \egaldistr \left( \lambda^\frac{1}{\alpha} Y_0(t/\lambda), t \geq 0 \right).
\end{equation}

We first prove that $\lim_{t \to 0} \frac{\Delta(t)}{t^{1 + \frac{1}{\alpha}}} = 0$ in probability. There exists $K_\alpha > 0$ such that $|(1 + x)^{1-\alpha}-1 -(1-\alpha)x - (1-\alpha)x^2| \leq K_\alpha x^2$ for any $x \in (0,1)$. Therefore, on the event $\calA_{t,\epsilon}$, for any $s \leq t$, we have
\begin{align*}
  |\Delta(s)| &\leq \int_0^{U(s)} \frac{\left|(1 + Y_0(r))^{1-\alpha}-1 -(1-\alpha)Y_0(r) - (1-\alpha)Y_0(r)^2 \right|}{Y(r)}dr\\
  &\leq \frac{K_\alpha}{1-\epsilon} \int_0^{(1+\epsilon)s} Y_0(r)^2 dr. 
\end{align*}
Using \eqref{eqn:scaling} with $\lambda=t$, for any $\delta > 0$, we have
\begin{align*}
  \P( |\Delta(t)| \geq \delta t^{1 + \frac{1}{\alpha}} ) &\leq \P( \calA_{t,\epsilon}^c) + \P\left( \frac{K_\alpha t^{1 + \frac{2}{\alpha}}}{1-\epsilon} \int_0^{1+\epsilon} Y_0(r)^2 \geq \delta t^{1 + \frac{1}{\alpha}} \right)\\
  &\leq C t \epsilon^{-\alpha} + \P\left( \frac{K_\alpha}{1-\epsilon} \int_0^{1+\epsilon} Y_0(r)^2 \geq \delta t^{-\frac{1}{\alpha}}\right).
\end{align*}
Letting $t \to 0$, we have $\lim_{t \to 0} t^{-1 - \frac{1}{\alpha}} \Delta(t) = 0$ in probability.

We now study the asymptotic behaviour of $t^{-1-\frac{1}{\alpha}}\int_0^{U(t)} Y_0(s)ds$. First observe that for any $\delta,\eta > 0$, we have
\begin{align*}
  &\P\left( \left| \int_t^{U(t)} Y_0(s) ds \right| \geq \eta t^{1+\frac{1}{\alpha}} \right)\\
  \leq &\P\left( |U(t)-t| \geq \delta t \right) + \P\left( \int_{(1-\delta)t}^{(1+\delta)t} |Y_0(s)| ds \geq \eta t^{1+\frac{1}{\alpha}} \right)\\
  \leq &\P\left( \left| \frac{U(t)}{t} - 1 \right| \geq \delta \right) + \P\left( \int_{1-\delta}^{1+\delta} |Y_0(s)| ds \geq \eta \right),
\end{align*}
using \eqref{eqn:scaling}. As $\lim_{t \to 0} \frac{U(t)}{t} = 1$ a.s, letting $t \to 0$ then $\delta \to 0$, we conclude that
\[
  \lim_{t \to 0} \int_t^{U(t)} Y_0(s) ds = 0 \quad \text{in probability}.
\]
Finally, using \eqref{eqn:scaling} again, we have $t^{-1-\frac{1}{\alpha}} \int_0^t Y_0(s) ds \egaldistr \int_0^1 Y_0(s) ds = X$ for any $t >0$.
As a conclusion, \eqref{eqn:firstDecomposition} yields $\displaystyle \lim_{t \to 0} t^{-1-\frac{1}{\alpha}} \tilde{R}_t = (1-\alpha)C_\alpha X$ in law.
\end{proof}

As a straightforward consequence of Theorem \ref{thm:asymptoticR}, we obtain the asymptotic behaviour of $R^{-1}$ at small times.
\begin{corollary}
\label{cor:Rm1}
We have $\lim_{t \to 0} \frac{R^{-1}(t) - \frac{t}{C_\alpha}}{t^{1 + \frac{1}{\alpha}}} = \frac{(\alpha - 1)}{C_\alpha^{1 + \frac{1}{\alpha}}} X$ in law.
\end{corollary}

\begin{proof}
Let $x \in \R$ and $t \geq 0$, we observe that
\[
  \P\left( R^{-1}(t) - \frac{t}{C_\alpha} > t^{1 + \frac{1}{\alpha}}x \right) = \P\left( R(\tau_{x,t}) < t \right),
\]
where we set $\tau_{x,t} = \frac{t}{C_\alpha} + t^{1 + \frac{1}{\alpha}} x$. Observe that for any fixed $x \in \R$, we have
\[
  t = C_\alpha \tau_{x,t} - x C_\alpha^{2 + \frac{1}{\alpha}} \tau_{x,t}^{1 + \frac{1}{\alpha}} + o(\tau_{x,t}^{1 + \frac{1}{\alpha}}),
\]
as $t \to 0$. Therefore, by Theorem \ref{thm:asymptoticR}, we obtain
\begin{align*}
  \qquad \qquad \lim_{t \to 0}  \P\left( R^{-1}(t) - \frac{t}{C_\alpha} > t^{1 + \frac{1}{\alpha}}x \right) &= \P\left( (1-\alpha)C_\alpha X < - x C_\alpha^{2 + \frac{1}{\alpha}} \right)\\
  &= \P\left(\frac{(\alpha-1)X}{C_\alpha^{1 + \frac{1}{\alpha}}} > x \right). \qquad \qquad \text{\qedhere}
\end{align*}
\end{proof}

Using this result, we now compute the asymptotic behaviour of $R^{-1}(t)^{-\frac{1}{\alpha - 1}}$, which is used to prove Theorem \ref{thm:main}.
\begin{lemma}
\label{lem:estimateRpower}
We denote by $D_\alpha = \frac{(\alpha \Gamma(\alpha))^{\frac{1}{\alpha(\alpha-1)}}}{(\alpha-1)^\frac{1}{\alpha}}$, we have
\[
  \lim_{t \to 0} t^{\frac{1}{\alpha(\alpha-1)}} \left( \left((\alpha - 1) R^{-1}(t)\right)^{-\frac{1}{\alpha-1}} - \left(\alpha \Gamma(\alpha)/t\right)^\frac{1}{\alpha - 1} \right) = -D_\alpha X \quad \text{in law.}
\]
\end{lemma}

\begin{proof}
The proof follows the same lines as Corollary \ref{cor:Rm1}. For any $x \in \R$, for any $t > 0$ small enough we have
\begin{align*}
  &\P\left( \left((\alpha - 1) R^{-1}(t)\right)^{-\frac{1}{\alpha-1}} - \left(\alpha \Gamma(\alpha)/t\right)^\frac{1}{\alpha - 1} > xt^{-\frac{1}{\alpha(\alpha-1)}} \right)\\
  &= \P\left( (\alpha-1)R^{-1}(t) < \left(  \left(\alpha \Gamma(\alpha)/t\right)^\frac{1}{\alpha - 1} +  xt^{-\frac{1}{\alpha(\alpha-1)}}\right)^{1-\alpha} \right)\\
  &= \P\left( (\alpha-1)R^{-1}(t) < \frac{t}{\alpha \Gamma(\alpha)} + \frac{(1-\alpha)x}{(\alpha \Gamma(\alpha))^{\frac{\alpha}{\alpha-1}}}t^{1+\frac{1}{\alpha}} + o(t^{1+\frac{1}{\alpha}}) \right).
\end{align*}
Therefore, using Corollary \ref{cor:Rm1}, we obtain for any $x \in \R$
\begin{equation*}
  \lim_{t \to 0} \P\left( \left((\alpha - 1) R^{-1}(t)\right)^{-\frac{1}{\alpha-1}} - \left(\alpha \Gamma(\alpha)/t\right)^\frac{1}{\alpha - 1} > xt^{-\frac{1}{\alpha(\alpha-1)}} \right)=\P(D_\alpha X < - x),
\end{equation*}
which concludes the proof.
\end{proof}

\begin{proof}[Proof of Theorem \ref{thm:main}]
By Lemma \ref{lem:coupling}, the asymptotic behaviours of the number of blocs $N(t)$ and $D(R^{-1}(t))$ are the same. Therefore, we only have to prove that
\[
  \lim_{t \to 0} t^\frac{1}{\alpha(\alpha - 1)} \left(D(R^{-1}(t)) - \left(\alpha \Gamma(\alpha)/t\right)^{\frac{1}{\alpha - 1}}\right) = -D_\alpha X \quad \text{in law}.
\]

Observe that by Corollary \ref{cor:Rm1}, we have $\lim_{t \to 0} C_\alpha R^{-1}(t)/t = 1$ in probability. Moreover, as $\alpha \in (1,2)$, we have $\frac{1}{\alpha(\alpha-1)} > \frac{1}{2(\alpha-1)}$, thus 
\[\lim_{\tau \to 0} \frac{D(\tau) - \left((\alpha - 1) \tau\right)^{\frac{-1}{\alpha - 1}}}{\tau^{\frac{-1}{\alpha (\alpha - 1)}}} = 0 \quad \text{ a.s.}\]
by Theorem \ref{thm:poisson}. We conclude that
\[  
  \lim_{t \to 0} t^{\frac{1}{\alpha (\alpha - 1)}}\left(D(R^{-1}(t))- \left((\alpha - 1) R^{-1}(t)\right)^{\frac{-1}{\alpha - 1}}\right)  = 0 \quad \text{in probability.}
\]

Therefore, using Lemma \ref{lem:estimateRpower}, we have
\begin{align*}
  &\lim_{t \to 0} t^\frac{1}{\alpha(\alpha - 1)} \left(D(R^{-1}(t)) - \left(\alpha \Gamma(\alpha)/t\right)^{\frac{1}{\alpha - 1}}\right)\\
  = &\lim_{t \to 0} t^\frac{1}{\alpha(\alpha - 1)} \left( \left((\alpha - 1) R^{-1}(t)\right)^\frac{-1}{\alpha - 1} - \left(\alpha \Gamma(\alpha)/t\right)^{\frac{1}{\alpha - 1}}\right) = -D_\alpha X \quad \text{in law.}
\end{align*}
\end{proof}

\bibliographystyle{plain}

\end{document}